\documentclass[12pt,leqno]{amsart}
\usepackage{amsmath, amsfonts, amssymb, amsthm, mathrsfs}
\usepackage[papersize={7.6in,10.2in},textwidth=14.7cm,textheight=20.4cm,centering]{geometry}

\usepackage[colorlinks=true,citecolor=blue,linkcolor=blue]{hyperref}
\hypersetup{
pdfstartpage=1,
pdfstartview=FitH}

\def\1{\mathbf{1}}

\theoremstyle{plain}
\newtheorem{theorem}{Theorem}

\newtheorem{lemma}{Lemma}
\newtheorem{corollary}{Corollary}
\theoremstyle{definition}

\theoremstyle{remark}

\def\leq{\leqslant}

\def\geq{\geqslant}

\DeclareMathOperator{\Mod}{mod}

\renewcommand{\bmod}[1]{\,(\Mod{ #1})}

\begin{document}

\title{The large $k$-term progression-free sets in $\mathbb{Z}_q^n$}
\thanks{This work is supported by the National Natural Science Foundation of
China (Grant No. 11271249 and No. 11671253).}
\author{Hongze Li}
\address{Department of Mathematics, Shanghai Jiao Tong University, Shanghai 200240, People's Republic of China}
\email{lihz@sjtu.edu.cn}

\maketitle

\begin{center}
\dedicatory{\textit{ \small{In memory of Professor Chengdong Pan} }}
\end{center}

\begin{abstract}
Let $k$ and $n$ be fixed positive integers. For each prime power $q\geqslant k\geqslant 3$, we show that any subset $A\subseteq \mathbb{Z}_q^n$ free of
$k$-term arithmetic progressions has size $|A|\leqslant c_k(q)^n$ with a constant $c_k(q)$ that can be expressed explicitly in terms of $k$ and $q$.
As a consequence, we can take $c_k(q)=0.8415q$ for sufficiently large $q$ and arbitrarily fixed $k\geq 3$.
\end{abstract}

\section{\bf Introduction}
\setcounter{lemma}{0}\setcounter{theorem}{0}\setcounter{corollary}{0}
\setcounter{equation}{0}

In his famous papers \cite {Rot52},\cite{Rot53}, Roth first considered the problem of finding upper bounds for the size of large subset of $\{1,2,...,N\}$ with no three-term arithmetic progression, and gave the first nontrivial upper bound. Since then, this problem has received considerable attentions by number theorists. Let $r_3(N)$ denote the maximal size of a subset of $\{1,2,...,N\}$ with no three-term arithmetic progression. Roth indeed proved $r_3(N)=O(N/\log\log N)$. This was subsequently improved and enhanced by
Heath-Brown \cite{HB87}, Szemer\'{e}di \cite{Sze90}, Bourgain \cite{Bou99}, Sanders \cite{San11}, \cite{San12}, and Bloom \cite{Blo16}. The best result so far is $r_3(N)=O(N(\log\log N)^4/\log N)$, due to Bloom.

For an (additively written) abelian group $G$, we say that a subset $A$ of $G$ is $k$-term \textit{progression-free} if there do not exist $a_1,\,a_2,\ldots,a_k\in A$ such that $a_k-a_{k-1}=a_{k-1}-a_{k-2}=\ldots= a_2-a_1\neq 0$, and denote by $r_k(G)$ the maximal size of $k$-term progression-free subsets of $G$.

In \cite{BB82},  Brown and Buhler first proved that $r_3(\mathbb{Z}_3^n)=o(3^n)$, and this was quantified by Meshulam \cite{Mes95} to $r_3(\mathbb{Z}_3^n)=O(3^n/n)$.  In their ground-breaking paper, Bateman and Katz \cite{BK12} proved that $r_3(\mathbb{Z}_3^n)=O(3^n/n^{1+\eta})$ with some positive constant $\eta>0$. The best known upper bound, $o(2.756^n)$, is due to Ellenberg and Gijswijt \cite{EG17}. Especially, they proved that, for any prime $p\geqslant 3$, there exists a positive constant $c=c(p)<p$ such that $r_3(\mathbb{Z}_p^n) = o(c^n)$. For the upper bound of $r_3(\mathbb{Z}_4^n)$, Sanders \cite{San09} proved that $r_3(\mathbb{Z}_4^n)=O(4^n/n(\log n)^{\eta})$ with an absolute constant $\eta>0$. Quite recently, Croot, Lev and Pach \cite{CLP17} developed the polynomial method and drastically improved the above upper bound to $r_3(\mathbb{Z}_4^n)\leqslant 4^{0.926n}$ in their breakthrough paper.

For each positive integer $m$, define
\begin{align}\label{eq:A(m)}
\mathfrak{A}(m)=\min_{x\in (0,1)}\frac{(1-x^m)}{m(1-x)x^{\frac{m-1}{3}}}
.\end{align}

In this paper, we introduce a formal polynomial method and establish the following upper bound of $r_k(\mathbb{Z}_{p^{\alpha}}^n)$ for $p\geq 2$ and $k\geq 3$.

\begin{theorem}\label{mt}
For any prime powers $q=p^{\alpha}\geqslant k\geqslant 3$ and $n\geqslant 1$, we have
\[r_k(\mathbb{Z}_q^n)\leqslant \Big(q\cdot \mathfrak{A}\Big(\frac{q}{(L_k,q)}\Big)\Big)^n,\]
where $\mathfrak{A}(\cdot)$ is given by $\eqref{eq:A(m)}$ and $L_k$ denotes the l.c.m. of $2,3,\ldots,k-1$.
\end{theorem}

\begin{corollary}\label{coro}
~

\begin{enumerate}
\item
For $k\geq 3$ and large $q$, $r_k(\mathbb{Z}_q^n) \leqslant (0.8415q)^n$.
\item
For $k\geq 3$ and each $q>(L_k,q)$, $r_k(\mathbb{Z}_q^n) \leqslant (0.945q)^n$.
\end{enumerate}

\end{corollary}

\subsection*{Notation}
Throughout this paper, $p$ with or without subscript, is always reserved for primes.
Denote by $(a,b)$ the greatest
common divisor of $a$ and $b$, $L_t:=[2,3,...,t-1]$ the least common multiple of $2,3,\ldots,t-1$.
For a set $S$, denote by $|S|$ the cardinality of $S$, and define
$
m S:=\{m s:\,s\in S\}.
$

\subsection*{Acknowledgement} The author is very grateful to Ping Xi for his valuable suggestions and comments.

\section{\bf Some Lemmas}
\setcounter{lemma}{0}\setcounter{theorem}{0}\setcounter{corollary}{0}
\setcounter{equation}{0}

Throughout this section, we fix $n\geqslant 1$ and $q=p^\alpha\geqslant k\geqslant 3$.

Given a positive integer $m$, the unknown $Y$ is said to be a generator of order $m$, if $Y^0=Y^m=1$ and $Y^j\neq 1$ for $1\leqslant j\leqslant m-1$.
For $1\leqslant i\leqslant n$, let $Y_{i}$ be generators of order $q$, then
$Y_{i}^q=1,$
and
$$
\prod_{i=1}^nY_{i}^{\lambda_{i}}=1\text{ if and only if }\lambda_{i}\equiv 0\bmod{q}\text{ for each }1\leqslant i\leqslant n.
$$

For $1\leq i\leq n$, put $X_i=Y_i-1$.
Let $\mathbb{Z}_p[X_1,\ldots,X_n]$ denote the linear space spanned by monomials $\{X_1^{\lambda_1}\cdots X_n^{\lambda_n}:0\leqslant\lambda_i\in \mathbb{Z}\}$ with coefficients over
$\mathbb{Z}_p$. $F[X_1,\ldots,X_n]=0$ means all coefficients of $F[X_1,\ldots,X_n]$ is $0$ over $\mathbb{Z}_p$.
We thus have $Y_i^q=(X_i+1)^q=X_i^q+1,$ which gives $X_i^q=0$ since $Y_i$
is of order $q$. Hence it is reasonable to assume that the terms of $X_1^{\lambda_1}\cdots X_n^{\lambda_n}$ vanish if $\lambda_i\geqslant q$
for some $1\leqslant i\leqslant n.$

For $0<\alpha<1/2$, define
\begin{align}\label{eq:setM}
\mathfrak{M}_{\alpha,\,q}:=\left\{(\lambda_{1},\lambda_{2},\ldots,\lambda_{n})\in[0, q)^n\cap\mathbb{Z}^n:\displaystyle\sum\limits_{i=1}^n\frac{\lambda_{i}}{q-1}\leqslant  \alpha n\right\}\end{align}
and $\overline{\mathfrak{M}}_{\alpha,\,q}=([0, q)^n\cap\mathbb{Z}^n)\setminus\mathfrak{M}_{\alpha,\,q}$ denotes the complementary set.
It is clear that $|\overline{\mathfrak{M}}_{\alpha,\,q}|=|\mathfrak{M}_{1-\alpha,\,q}|.$

For $c=(c_1,\ldots,c_n)\in \mathbb{Z}_q^n$,
and $\mathbf{\mathbf{X}}=(X_{1},\ldots,X_{n})$, define
$$
\mathbf{\mathbf{X}}^c:=\prod_{i=1}^nX_{i}^{c_{i}}.
$$
For $a,b\in \mathbb{Z}_q^n$, one can define $a_{i},b_{i},\mathbf{\mathbf{X}}^a,\mathbf{\mathbf{X}}^b$ accordingly.

For a set $B\subseteq \mathbb{Z}_q^n$, let $V_B$ denote the sub-space spanned by $\{\mathbf{X}^a: a\in B\}$
over $\mathbb{Z}_p$. Then $\dim V_B=|B|.$ When $B=\mathbb{Z}_q^n$ is the whole space, we write $V_{\mathbb{Z}_q^n}=V$ and thus $\dim V=q^n.$ For each $f\in V$, we may write
\[f=\sum_{a\in \mathbb{Z}_q^n}f(a)\mathbf{X}^a\]
with coefficients $f(a),a\in\mathbb{Z}_q^n.$

\begin{lemma}\label{lm:basiclemma}
Suppose $k\geq 3$ and $0<\alpha<1/2$.
Let $A$ be a subset of $\mathbb{Z}_q^n$ satisfies $ra\neq rb$ for $a\neq b\in A$ with $1\leq r\leq k-1$.
Suppose $P\in V_{\mathfrak{M}_{2\alpha,q}}$
satisfies $P(2a-b)P(3a-2b)\cdots P((k-1)a-(k-2)b)=0$ for every pair $a,\,b$ of distinct elements in $A$. Then there exists an element $c\in A$ such that $P(c)=0$ when $|A|>2^{k-2}\,|\mathfrak{M}_{\alpha,q}|.$
\end{lemma}

\begin{proof}

For brevity, we only prove the lemma for $k=4$, and the method also works in the general case.

For $0<\alpha<1/2$,
put
$
m_{\alpha}=\big\{\mathbf{X}^{\lambda}:\lambda\in \mathfrak{M}_{\alpha,q}\big\},
$
so we can write
$$
P(\mathbf{X})=\mathop{\sum\sum}_{fg\in m_{2\alpha}}c_{f,g}f(\mathbf{X})g(\mathbf{X}).
$$
In each term of the summand, at least one of $f$ and $g$ is in $m_{\alpha}$.
Hence
$$
P(\mathbf{X})=\sum_{f\in m_{\alpha}}f(\mathbf{X})F_f(\mathbf{X})+\sum_{g\in m_{\alpha}}g(\mathbf{X})G_g(\mathbf{X}).
$$
We thus have
\begin{align*}
P(\mathbf{X})P(\mathbf{Y})&=\mathop{\sum\sum}_{f,f_1\in m_{\alpha}}f(\mathbf{X})f_1(\mathbf{Y})F_f(\mathbf{X})F_{f_1}(\mathbf{Y})
\\
&\ \ \ \ \ \ +\mathop{\sum\sum}_{g,g_1\in m_{\alpha}}g(\mathbf{X})g_1(\mathbf{Y})G_g(\mathbf{X})G_{g_1}(\mathbf{Y})
\\
&\ \ \ \ \ \ +\mathop{\sum\sum}_{f,g_1\in m_{\alpha}}f(\mathbf{X})g_1(\mathbf{Y})F_f(\mathbf{X})G_{g_1}(\mathbf{Y})
\\
&\ \ \ \ \ \ +\mathop{\sum\sum}_{f_1,g\in m_{\alpha}}f_1(\mathbf{X})g(\mathbf{Y})F_{f_1}(\mathbf{X})G_{g}(\mathbf{Y})
\end{align*}
for some families of polynomials $F,\,G$ indexed by $m_{\alpha}$.

Write
$A=\{a_1,a_2,\ldots,a_t\}.$
Now let $B$ be the $t\times t$ matrix whose $i,j$ entry is $P(2a_i-a_j)P(3a_i-2a_j)$. Then
\begin{align*}
B_{ij}&=\mathop{\sum\sum}_{f,f_1\in m_{\alpha}}f(2a_i)f_1(3a_i)F_f(-a_j)F_{f_1}(-2a_j)\\
&\ \ \ \ \ \ +\mathop{\sum\sum}_{g,g_1\in m_{\alpha}}G_g(2a_i)G_{g_1}(3a_i)g(-a_j)g_1(-2a_j)
\\
&\ \ \ \ \ \ +\mathop{\sum\sum}_{f,g_1\in m_{\alpha}}f(2a_i)G_{g_1}(3a_i)F_f(-a_j)g_1(-2a_j)
\\
&\ \ \ \ \ \ +\mathop{\sum\sum}_{f_1,g\in m_{\alpha}}G_g(2a_i)f_1(3a_i)g(-a_j)F_{f_1}(-2a_j)\\
&=B_{ij}^{(1)}+B_{ij}^{(2)}+B_{ij}^{(3)}+B_{ij}^{(4)},
\end{align*}
say.
Hence $(B_{ij}^{(s)})$ is a sum of at most $|\mathfrak{M}_{\alpha,q}|^2$ matrices for each $s$.
One may see that each matrix in $(B_{ij}^{(1)})$ has the form
$$
\begin{pmatrix}
f(2a_1)f_1(3a_1)\\f(2a_2)f_1(3a_2)\\ \cdots \\f(2a_t)f_1(3a_t)
\end{pmatrix}
(F_f(-a_1)F_{f_1}(-2a_1),\,\,F_f(-a_2)F_{f_1}(-2a_2),\cdots,F_f(-a_t)F_{f_1}(-2a_t))
$$
and of rank $1$ or $0$; the rank is 0 unless there exists some $a_i$ such that $f(2a_i)f_1(3a_i)=1$, and the number of such $a_i$ is at most $|\mathfrak{M}_{\alpha,q}|$. This yields the rank of $(B_{ij}^{(1)})$ is at most $|\mathfrak{M}_{\alpha,q}|$. Similarly, one can show that the rank of $(B_{ij}^{(2)})$ is also at most $|\mathfrak{M}_{\alpha,q}|$.

Regarding $(B_{ij}^{(3)})$, each of the $|\mathfrak{M}_{\alpha,q}|^2$ matrices has the form
$$
\begin{pmatrix}
f(2a_1)G_{g_1}(3a_1)\\f(2a_2)G_{g_1}(3a_2)\\ \cdots \\f(2a_t)G_{g_1}(3a_t)
\end{pmatrix}
(F_f(-a_1)g_1(-2a_1),\,\,F_f(-a_2)g_1(-2a_2),\cdots,F_f(-a_t)g_1(-2a_t))
$$
and of rank $1$ or $0$; the rank is 0 unless there exists some $a_i$ and $a_j$ such that $f(2a_i)=1$ and $g_1(-2a_j)=1$, then this matrix has only one non-zero element. The number of such $a_i$ is at most $|\mathfrak{M}_{\alpha,q}|$, hence the row rank of $(B_{ij}^{(3)})$ is also at most $|\mathfrak{M}_{\alpha,q}|$, which also applies similarly to $(B_{ij}^{(4)})$. Thus the rank of $B$ is at most $4|\mathfrak{M}_{\alpha,q}|$.

On the other hand, by the hypothesis on $P$, $B$ must be a diagonal matrix. This completes the proof.
\end{proof}

\begin{lemma}\label{lm:keylemma}
Let $q\geqslant k\geqslant 3$,
$A$ a subset of $\mathbb{Z}_{q}^n$ which
doesn't contain $k$-term arithmetic progressions. Then we have
\begin{align}\label{eq:keylemma}
|A|\leqslant (2^{k-2}+1)
d^n
|\mathfrak{M}_{1/3,q/d}|,
\end{align}
where $d=(L_k,q).$
\end{lemma}

\begin{proof}

Suppose $t\geq 3$ and $b_1,\,b_2,\ldots,b_t$ is a non-trivial $t$-term arithmetic progression, then
$$
b_t-b_{t-1}=b_{t-1}-b_{t-2}=\cdots= b_2-b_1\neq 0,
$$
and
\begin{equation}
b_j=b_1+(j-1)(b_2-b_1)=(j-1)b_2-(j-2)b_1,\,\,\,\,\,3\leq j\leq t.
\end{equation}

Hence each non-trivial $t$-term arithmetic progression $b_1,\,b_2,\ldots,b_t$ is determined by $b_2$ and $b_1$ only, taking the order into account.

Let $F$ be the kernel of the homomorphism of $\mathbb{Z}_q^n$ defined by
$g\mapsto \frac{q}{(L_k,q)}g\,(g\in \mathbb{Z}_q^n)$,
then
$$
F=(L_k,q)\mathbb{Z}_q^n\cong \mathbb{Z}_{\frac{q}{(L_k,q)}}^n,
$$
and
$$
\mathbb{Z}_q^n/F\cong \mathbb{Z}_{(L_k,q)}^n.
$$

Let $\mathfrak{R}$ be the set of all $F$-cosets, we write
$\mathfrak{R}=\{R_1,R_2,...,R_{(L_k,q)^n}\}$.
For $1\leqslant j\leqslant (L_k,q)^n$, let $A_j:=A\bigcap R_j$, we choose one element $r_j\in A_j$, and then we have
$$R_j=r_j+F.$$

Without loss of generality, we consider $A_1$.
First, we have
$$
(A_1-r_1)\bigcap R_1\subseteq F=(L_k,q)\mathbb{Z}_q^n\cong \mathbb{Z}_{\frac{q}{(L_k,q)}}^n.
$$

Therefore $A_1$ doesn't contain any $k$-term arithmetic progression.
Define $B$ by
$$
A_1-r_1=(L_k,q)B,\,\,\,
B\subseteq \mathbb{Z}_q^n.
$$
Hence $B$ doesn't contain any $k$-term arithmetic progression and satisfies $ra\neq rb$ for $a\neq b\in B$ with $1\leq r\leq k-1$.

We shall prove that $|B|\leqslant (2^{k-2}+1)|\mathfrak{M}_{1/3,q/(L_k,q)}|$,
which would yield
\begin{equation}
|A|\leqslant (L_k,q)^n|B|\leqslant (2^{k-2}+1)(L_k,q)^n|\mathfrak{M}_{1/3,q/(L_k,q)}|.
\end{equation}

Assuming, contrary to what we want to prove, that $|B|> (2^{k-2}+1)|\mathfrak{M}_{1/3,q/(L_k,q)}|$.
Let $W$ denote the linear space spanned by $\{\mathbf{X}^{\lambda}:\lambda\in B\bigcap{\mathfrak{M}}_{2/3,\,q/(L_k,q)}\}$,
then
\begin{align*}
\dim W&\geqslant |{\mathfrak{M}}_{2/3,\,q/(L_k,q)}|+|B|-(q/(L_k,q))^n\\
&=|B|-\{(q/(L_k,q))^n-|{\mathfrak{M}}_{2/3,\,q/(L_k,q)}|\}\\
&=|B|-|{\overline{\mathfrak{M}}}_{2/3,\,q/(L_k,q)}|\\
&=|B|-|\mathfrak{M}_{1/3,q/(L_k,q)}|\\
&>2^{k-2}|\mathfrak{M}_{1/3,q/(L_k,q)}|.
\end{align*}
We can choose some $b_i\in B$ such that $$P:=\mathbf{X}^{b_1}
+\mathbf{X}^{b_2}+\cdots+\mathbf{X}^{b_t}
\in W\subseteq V_{{\mathfrak{M}}_{2/3,\,q/(L_k,q)}}.$$

Let
$$
B_1:=\{b_1,b_2,\ldots,b_t\}\subseteq B.
$$

By assumption we have $P(2a-b)P(3a-2b)\cdots P((k-1)a-(k-2)b)=0$ for every pair $a,\,b$ of distinct elements in $B_1$. Taking $\alpha=1/3$ in Lemma \ref{lm:basiclemma}, we have $P(b_i)=0$ for some $b_i\in B_1$, this is a contradiction. Hence
$$
|B|\leqslant (2^{k-2}+1)|\mathfrak{M}_{1/3,q/(L_k,q)}|,
$$
and the lemma follows.
\end{proof}

\begin{lemma}\label{lm:M-bound}
We have
$$
|\mathfrak{M}_{1/3,\,q}|\leqslant q^n\mathfrak{A}(q)^n.
$$
\end{lemma}

\begin{proof}
Write $\xi_{i}=\lambda_{i}/(q-1)$, and we regard $\xi_{i}$ as random variables uniformly distributed in the set
$$
\Big\{0,\frac{1}{q-1},\,\frac{2}{q-1},\ldots,\frac{q-1}{q-1}\Big\}.
$$
Then
\begin{equation}
\frac{|\mathfrak{M}_{1/3,\,q}|}{q^n}
=\mathbf{Pr}\Big(\sum_{i=1}^n\xi_{i}\leq n/3 \Big)=\mathbf{Pr}\Big(x^{\sum_{i=1}^n\xi_{i}}\geq x^{n/3}\Big)\end{equation}
for any $x\in (0,1)$.
By Chernoff bound, we have
\begin{align*}
\frac{|\mathfrak{M}_{1/3,\,q}|}{q^n}&\leq
x^{-n/3}\mathbf{E}\Big[x^{\sum_{i=1}^n\xi_{i}}\Big]
=\Big(\prod_{i=1}^nx^{-1/3}\mathbf{E}[x^{\xi_{i}}]\Big).
\end{align*}
On the other hand, from the uniform distribution of $\xi_i$ ($1\leqslant i\leqslant n),$ it follows that
$$
x^{-1/3}\mathbf{E}[x^{\xi_i}]=\frac{1}{qx^{1/3}}\sum_{j=0}^{q-1}x^{\frac{j}{q-1}}
=\frac{1-y^q}{q(1-y)y^{\frac{q-1}{3}}}
$$
with $y=x^{\frac{1}{q-1}}$.

Hence we may conclude that
$$
\frac{|\mathfrak{M}_{1/3,\,q}|}{q^n}\leq \Big(\frac{1-y^q}{q(1-y)y^{\frac{q-1}{3}}}\Big)^n.
$$
The lemma then follows from the arbitrariness of $x$ (and thus of $y$).
\end{proof}

\section{\bf Proof of Theorem \ref{mt}}
\setcounter{lemma}{0}\setcounter{theorem}{0}\setcounter{corollary}{0}
\setcounter{equation}{0}

Now we give the proof of Theorem \ref{mt}.
\begin{proof}

For $q=p^\alpha$, let $A$ be a subset of $\mathbb{Z}_q^n$ free of
$k$-term arithmetic progressions.
By Lemmas \ref{lm:keylemma} and \ref{lm:M-bound}, for $q> (L_k,q)$ we have
\begin{align*}
|A|&\leqslant (2^{k-2}+1)\cdot (L_k,q)^n|\mathfrak{M}_{1/3,\,\frac{q}{(L_k,q)}}|\leqslant
(2^{k-2}+1)\cdot q^n\cdot \mathfrak{A}\Big(\frac{q}{(L_k,q)}\Big)^n.
\end{align*}
For each positive integer $v$, using the tensor trick, the set $A\times A\times\cdots\times A\subseteq \mathbb{Z}_q^{vn}$ is $k$-term progression-free, and therefore
$$
|A|^v\leqslant
(2^{k-2}+1)\cdot q^{vn}\cdot \mathfrak{A}\Big(\frac{q}{(L_k,q)}\Big)^{vn}.
$$
This implies Theorem \ref{mt} by letting $v$ approach to infinity.
\end{proof}

\section{\bf Proof of Corollary \ref{coro}}
\setcounter{lemma}{0}\setcounter{theorem}{0}\setcounter{corollary}{0}
\setcounter{equation}{0}

Now we give the proof of Corollary \ref{coro}. Here $q$ is not necessary to be a prime power and we thus suppose $q=\prod_{i=1}^lp_i^{\alpha_i}$ as the standard factorization of $q$.

\begin{proof}

(1) For $q\rightarrow +\infty$, we must have $$M:=\max_{1\leq i\leq l}p_i^{\alpha_i}\rightarrow +\infty.$$

For large $M$, taking $x=1-\frac{\alpha}{M}$ with $\alpha=2.148$, we then have
$$
\lim_{M\rightarrow +\infty}\frac{1-x^M}{M(1-x)x^{\frac{M-1}{3}}}=
\lim_{M\rightarrow +\infty}\frac{1-(1-\frac{\alpha}{M})^M}{\alpha(1-\frac{\alpha}{M})^{\frac{M-1}{3}}}
=\frac{e^{\alpha/3}-e^{-2\alpha/3}}{\alpha}
<0.8415,
$$ which yields $\mathfrak{A}(M)\leq0.8415$ for all sufficiently large $M$. It follows  that
$$
r_k(\mathbb{Z}_q^n)\leq (q/M)^nr_k(\mathbb{Z}_M^n) \leq (q/M)^n M^n\mathfrak{A}(M)^n\leq (0.8415q)^n.
$$

(2) For $x=1-\frac{\beta}{N}$, $\beta=1.6$, we have
$$
\frac{1-x^N}{N(1-x)x^{\frac{N-1}{3}}}=
\frac{1-(1-\frac{\beta}{N})^N}{\beta(1-\frac{\beta}{N})^{\frac{N-1}{3}}}=
\frac{(1-\frac{\beta}{N})^{-\frac{N-1}{3}}-(1-\frac{\beta}{N})^{\frac{2N+1}{3}}}{\beta}.$$
When $N\geq 13$, the above quantity is at most
$$
\frac{e^{\frac{N-1}{3}\frac{\beta/N}{1-\beta/N}}-e^{-\frac{2N+1}{3}\frac{\beta/N}{1-\beta/N}}}{\beta}
\leq \frac{e^{\frac{4\beta}{13-\beta}}-e^{-\frac{9\beta}{13-\beta}}}{\beta}
<0.92.
$$
On the other hand, for all prime powers $N<13$, we have the following list of explicit bounds for $\mathfrak{A}(N)$:
$$
\mathfrak{A}(2)<0.94495,\,\,\mathfrak{A}(3)<0.9184,\,\,\mathfrak{A}(4)<0.9027,\,\,\mathfrak{A}(5)<0.8924,\,\,
$$
$$
\mathfrak{A}(7)<0.8795,\,\,\mathfrak{A}(8)<0.8753,\,\,\mathfrak{A}(9)<0.8718,\,\,\mathfrak{A}(11)<0.8667.
$$
Hence we may state, for each prime power $N\geqslant2$, that
\begin{equation}\label{eq:boundA(N)}
\mathfrak{A}(N)<0.945.
\end{equation}

For $q>(L_k,q)$, there exists some prime power $p^\alpha\|\frac{q}{(L_k,q)}$.
We may apply \eqref{eq:boundA(N)} with $N=p^\alpha,$ getting
$$
r_k(\mathbb{Z}_q^n)\leq (q/N)^nr_k(\mathbb{Z}_{N}^n) \leq (q/N)^n N^n\mathfrak{A}(N)^n\leq (0.945q)^n.
$$
This establishes Corollary \ref{coro}.

\end{proof}

\end{document}